\documentclass[12pt]{article}
\usepackage{pstricks,subfigure,graphicx}
\usepackage{tikz}
\usetikzlibrary{snakes,arrows}
\usepackage{amsmath,amssymb}
\usepackage{tikz}
\usepackage{color}
\newtheorem{theorem}{Theorem}[section]
\newtheorem{lemma}[theorem]{Lemma}

\newtheorem{proposition}[theorem]{Proposition}

\newtheorem{remark}[theorem]{Remark}

\newcommand{\proof}{\noindent{\bf Proof.\ }}
\newcommand{\qed}{\hfill $\square$ \bigskip}
\newcommand{\GG}{G\overline{G}}

\DeclareMathOperator {\c_m} {c_m}
\DeclareMathOperator {\h_m} {h_m}
\DeclareMathOperator {\m} {m}

\DeclareMathOperator {\diam} {diam}

\let\deg\relax
\DeclareMathOperator {\deg} {deg}

\begin{document}

\title{On the monophonic convexity in complementary prisms}	
		
\author{
	Neethu P. K.$^1$
	\and
	Ullas Chandran S. V.$^1$
        \and
        Julliano R. Nascimento$^2$
	}

\date{\today}

\maketitle

\begin{center}
    $^1$Department of Mathematics, Mahatma Gandhi College, University of Kerala, Kesavadasapuram,  Thiruvananthapuram-695004, Kerala, India \\
	{\tt p.kneethu.pk@gmail.com}, {\tt svuc.math@gmail.com} \\
    $^2$Instituto de Informática, Universidade Federal de Goiás, Goiânia, GO, Brazil
        {\tt jullianonascimento@inf.ufg.br}
	\medskip

\end{center}

\begin{abstract}
A set $S$ of vertices of a graph $G$ is \emph{monophonic convex} if $S$ contains all the vertices belonging to any induced path connecting two vertices of $S$. The cardinality of a maximum proper monophonic convex set of $G$ is called the \emph{monophonic convexity number} of $G$. 
The \emph{monophonic interval} of a set $S$ of vertices of $G$ is the set $S$ together with every vertex belonging to any induced path connecting two vertices of $S$. The cardinality of a minimum set $S \subseteq V(G)$ whose monophonic interval is $V(G)$ is called the \emph{monophonic number} of $G$.
The \emph{monophonic convex hull} of a set $S$ of vertices of $G$ is the smallest monophonic convex set containing $S$ in $G$. The cardinality of a minimum set $S \subseteq V(G)$ whose monophonic convex hull is $V(G)$ is called the \emph{monophonic hull number} of $G$.
The \emph{complementary prism} $\GG$ of $G$ is obtained from the disjoint union of $G$ and its complement $\overline{G}$ by adding the edges of a perfect matching between them. In this work, we determine the monophonic convexity number, the monophonic number, and the monophonic hull number of the complementary prisms of all graphs. 
\end{abstract}

\medskip\noindent

{\bf AMS Subj.\ Class:} 05C12; 05C69.

\section{Introduction}
\label{sec:intro}

A \emph{finite convexity space} is a pair $(V,{\cal C})$ consisting of a non-empty finite set $V$ and a collection ${\cal C}$ of subsets of $V$ such that $\emptyset \in {\cal C}$, $V\in {\cal C}$ and ${\cal C}$ is closed under intersections. The elements of $\mathcal{C}$ are called \emph{convex sets}, see \cite{van-1993}. Different convexities associated with the vertex set of a graph are well-known. The most natural convexities in graphs are path convexities defined by a family of paths $\mathcal{P}$, in a way that a set $S$ of vertices of $G$ is  $\mathcal{P}$-\emph{convex} if $S$ contains all vertices of every path of $\mathcal{P}$ between vertices of $S$. An extensive survey of different types of path convexities can be found in \cite{pelayo-2015}. The well-known \emph{geodetic convexity} has $\mathcal{P}$ as the family of all shortest paths \cite{caceres-2005, buckley-1990, farber-1986}.
 
In this paper, we consider the \emph{monophonic convexity} in graphs. In this convexity, $\cal P$ is based on the family of induced paths. A \emph{chord} of a path $P: u_0,u_1,\dots,u_n$ is an edge $u_iu_j$, with $j \geq i + 2$. A $u,v$-path $P$ is  \emph{monophonic} (\emph{$m$-path}, for short) or \emph{induced} if it is a chordless path. A set of vertices $S$ of a graph $G$ is a \emph{monophonic convex set} ($m$-\emph{convex set} for short) if $S$ contains all vertices belonging to any induced path connecting two vertices of $S$. The cardinality of a maximum proper convex set of $G$ is called the \emph{monophonic convexity number} ($m$-\emph{convexity number} for short) of $G$, denoted by $\c_m(G)$.

The {\em monophonic interval} $J_G[u,v]$ between $u$ and $v$ is the set of vertices of all $u,v$-monophonic paths in $G$. For $S\subseteq V(G)$ we fix $J_G[S]=\bigcup_{_{u,v\in S}}J_G[u,v]$. Thus, a set $S$ is \emph{monophonic convex} if $J_G[S]=S$.
A set $S\subseteq V(G)$ is called a \emph{monophonic set} ($m$-set, for short) if $J_G[S] = V(G)$. The \emph{monophonic number} of $G$, denoted by $\m(G)$, is the cardinality of a minimum monophonic set of $G$.

The \emph{$m$-convex hull} $[S]_G$ of a set $S$ in $G$ is the smallest convex set containing $S$ in $G$. The convex hull $[S]_G$ can be formed from the sequence $J^{p}_{G}[S]$, where $p$ is a non-negative integer, $J^0_{G}[S]=S$, $J^{1}_{G}[S]=J_G[S]$, and $J^{p}_{G}[S]=J_{G}[J_{G}^{p-1}[S]]$, for $p\geq 2$.  
A set $S \subseteq V(G)$ is called an \emph{$m$-hull set} if $[S]_G = V(G)$. The  \emph{$m$-hull number} of $G$, denoted by $\h_m(G)$, is the cardinality of a minimum $m$-hull set of $G$.

The monophonic convexity was introduced in 1982, by Jamison~\cite{jamison-1982}. For additional topics of interest concerning to the monophonic convexity we refer to the reader \cite{caceres-2005, dourado-2010, dragan-1999, duchet-1988, eurinardo-2015}. In particular, Dourado, Protti and Szwarcfiter~\cite{dourado-2010} proved that the decision problems associated with the monophonic convexity number and the monophonic number are NP-complete on arbitrary graphs. On the other hand, they prove that the monophonic hull number of an arbitrary graph can be computed in polynomial time.

 If $G$ is a graph and $\overline{G}$ its complement, then the {\em complementary prism} $G\overline{G}$ of $G$ is the graph formed from the disjoint union of $G$ and $\overline{G}$ by adding the edges of a perfect matching between the corresponding vertices of $G$ and $\overline{G}$, see ~\cite{haynes-2007,haynes-2009}. For example,  $C_5\overline{C}_5$ is the Petersen graph. The geodesic convexity number of complementary prisms has been investigated recently in {\rm\cite{castonguay-2019}}. In {\rm\cite{castonguay-2019}}, the authors proved that the decision problem associated to the convexity number is NP-complete even restricted to complementary prisms $\GG$ when both $G$ and its complement $\overline{G}$ are connected. In the same work, the authors found the geodetic convexity number of $G\overline{G}$ when $G$ or $\overline{G}$ is disconnected. In \cite{neethu-2022}, the authors found a formula for the geodesic convexity number of complementary prisms of all trees, a rich subclass of the above-stated class. Other related topics studied recently on complementary prisms include geodetic hull number \cite{castonguay-2021, coelho-2022}, general position number \cite{neethu-2020}, and domination \cite{haynes-2009}. To our knowledge, there is no literature on the monophonic convexity number of the complementary prisms. In this paper, we derive formulas for the monophonic convexity number, monophonic number, and monophonic hull number of the complementary prisms of all graphs. In Section~\ref{sec:preliminary} we fix the notations, terminologies and present some initial remarks. In Sections~\ref{convexity},~\ref{monophonic}, and~\ref{hull} we discuss our results on the monophonic convexity number, monophonic number, and monophonic hull number, respectively.
		
\section{Preliminaries}
\label{sec:preliminary}

Graphs in this paper are finite, simple, and undirected. For basic graph terminologies, we follow \cite{chartrand-2006}. Let $G = (V(G), E(G))$ be a graph. For a vertex $v\in V(G)$, the \emph {open neighborhood} $N_G(v)$
is the set of neighbors of $v$, while the \emph {closed neighborhood} $N_G[v]=N_G(v)\cup\{v\}$.  The
\emph{degree} of $v$ is $deg_G(v)=|N_G(v)|$. A vertex $v$ is \emph{simplicial} if $N_G(v)$ induces a clique.  For $S\subseteq V(G)$, we fix $N_G[S] = \bigcup _{v\in S} N_G[v]$; and $N_G(S) = (N_G[S])\backslash S$. A vertex $v \in V(G)$ is \emph{pendant} if $deg_G(v) = 1$. The \emph{distance} between two vertices $u$ and $v$, denoted $d_G(u,v)$, is the length of a shortest path joining them if any; otherwise $d_G(u,v) = \infty$. The \emph{eccentricity} of a vertex $v \in V(G)$, denoted $ecc_G(v)$, is the maximum distance from $v$ to any vertex in $G$. We may simplify the notations above by omitting the index $G$ whenever $G$ is clear from the context. The \emph{diameter} of a graph $G$, denoted $\diam(G)$, is the maximum eccentricity among the vertices of $G$.
The \emph{order} of a graph $G$ is the number of vertices in $G$, denoted by $n(G)$. A connected component $G'$ in a disconnected graph $G$ is called \emph{trivial} if $n(G') = 1$, and \emph{non-trivial}, otherwise.
The \emph{independent number} $\alpha(G)$  is the cardinality of a largest set of nonadjacent vertices of $G$. Whereas
the \emph{clique number}  $\omega(G)$ is the cardinality of a largest clique of $G$. For a set $S\subseteq V(G)$, $S^c$ denotes the compliment of the set $S$ in $G$. i.e., $S^c=V(G)\setminus S$.

If $u\in V(G)\cap V(G\overline{G})$, then the unique neighbour of $u$ in $V(G\overline{G}) \cap V(\overline{G})$ is denoted by $\overline{u}$, called the {\em partner} of $u$ in $\overline{G}$. If $S\subseteq V(G)$, then the set of all partners of the vertices from $S$ is denoted by $\overline{S}$. For a positive integer $k$, we use the notation $[k] = \{1,2, \dots, k\}$.

Some remarks that arise from monophonic convexity definitions can be found in~\cite{dourado-2010}. The next three are useful for our purposes.

\begin{remark}[\cite{dourado-2010}]\label{rmk:simplicial_convexity}
Every graph with a simplicial vertex satisfies $\c_m(G)=n-1$.
\end{remark}

\begin{remark}[\cite{dourado-2010}]\label{rmk:simplicial}
Every monophonic set or $m$-hull set of a graph $G$ must contain all of its simplicial vertices.
\end{remark}

\begin{remark}[\cite{dourado-2010}]\label{rmk:hm_bound}
For any graph $G$, $\h_m(G) \leq \min\{\m(G), \c_m(G)+ 1\}$. 
\end{remark}

Furthermore, when dealing with monophonic intervals on complementary prisms, the next proposition will be useful.

\begin{proposition}\label{prop:monophonic_p4}
Let $G$ be a graph. If $P: u_1,u_2,u_3,u_4$ is an induced path in $G$, then $V(P) \cup V(\overline{P}) \subseteq J_{\GG}[u_1,u_4]$.    
\end{proposition}

\begin{proof}
Given that $u_1,u_2, \overline{u_2}, \overline{u_4}, u_4$ and $u_1, \overline{u_1}, \overline{u_3}, u_3, u_4$ are monophonic paths in $\GG$, the conclusion is direct.
\end{proof}


\section{Monophonic convexity number}
\label{convexity}

In this section, we determine the monophonic convexity number of complementary prisms $\GG$. First, we consider the case when both $G$ and $\overline{G}$ are connected, in Theorem~\ref{thm:convexity_connected}.

\begin{theorem}
\label{thm:convexity_connected}
Let $G$ be a connected graph such that $\overline{G}$ is connected, then $\c_m (\GG)=\max\{\omega(G),\alpha(G)\}$.
\end{theorem}
\proof Let $S$ be a monophonic convex set of $\GG$ with $|S|=\c_m (\GG)$ and fix $S_G=S\cap V(G)$ and $S_{\overline{G}}=S\cap V(\overline{G})$. Since $n(G)\geq 2$, we have that $|S|\geq 2$. Consider the following three cases. \\
\textbf{Case 1 :}
$S_G\neq \emptyset$ and $S_{\overline{G}}\neq \emptyset$. In this case, we prove that $|S|=2$. 

 First, we prove the following claim.\\
\textbf{Claim 1:} $S_{\overline{G}}=\overline{S_G}$. Let $u\in S_G$. Choose $v\in S_G$ such that $\overline{v}\in S_{\overline{G}}$. Let $x\in S_G$.  Let $\overline{P}$ be a $\overline{v},\overline{x}$-monophonic path in $\overline{G}$. Then $\overline{P}$ together with the edge $\overline{x}x$ is a $x,\overline{v}$-monophonic path containing $\overline{x}$ in $\GG$. So $\overline{x}\in S$. Hence $\overline{S_G}\subseteq S_{\overline{G}}$. On the other hand, let $\overline{y}\in S_{\overline{G}}$ and $P$ be a $u,y$-monophonic path in $G$. Then $P$ together with the edge $y\overline{y}$ is a $u,\overline{y}$--monophonic path containing $y$ in $\GG$. Hence $y\in S_G$. This shows that $S_{\overline{G}}\subseteq \overline{S_G}$.\\
\textbf{Claim 2:} $S_G$ is a clique or an independent set in $G$. If $|S_G|\leq 2$, then there is nothing to prove. So assume that $|S_G|\geq 3$. Suppose there exists $u,v\in S_G$ with $u$ and $v$ adjacent in $G$. Let $w$ be any vertex distinct from $u$ and $v$ in $G$. Suppose $w$ is nonadjacent to both $u$ and $v$. Then $\overline{u},\overline{w},\overline{v}$ is an monophonic path in $\GG$ containing $\overline{w}$. And hence by Claim 1, $w,\overline{w}\in S$. Also if $w$ is adjacent to exactly one of $u$ or $v$, say $u$, then $u,w,\overline{w},\overline{v}$ is an $u,\overline{v}$-monophonic path in $\GG$ containing both $w$ and $\overline{w}$. Hence $w,\overline{w}\in S$. The above observations show that
\begin{equation}
\label{neighbor}
 [N_G(u)\cap N_G(v)]^c\subseteq S_G.   
\end{equation}
 Moreover $N_G(u)\cap N_G(v)\neq \emptyset$, since $S\neq V(\GG)$. Let $w\in N_G(u)\cap N_G(v)$. We first show that any vertex $z\in S_G$ is adjacent to both $u$ and $v$. This in turn implies that $S_G$ is a clique, since $u$ and $v$ are arbitrary adjacent vertices in $S_G$. Suppose $z$ is non adjacent to $u$. First consider the case $zw\notin E(G)$. Then $v,w,\overline{w},\overline{z}$ is an $m$-path connecting $v$ and $\overline{z}$. This shows that $V(\GG)= S$. This is not possible as $S$ is a proper subset of $G$. Thus $zw\in E(G)$. In this case $z,w,u,\overline{u}$ is an $m$-path connecting $z$ and $\overline{u}$ containing $w$, and so $w\in S_G$. This again shows that $S=V(\GG)$. Which is not possible. Thus $uz\in E(G)$ and $vz\in E(G)$. Hence $S_G$ must be a clique in $G$ and $S_{\overline{G}}$ is an independent set in $\overline{G}$.\\
Now, if $S_G$ contains two nonadjacent vertices $u$ and $v$, then replacing $G$ by $\overline{G}$ in the above arguments we get that $S_{\overline{G}}$ is a clique. \\
Note that an independent set in a connected graph is $m$-convex only if it contains exactly one vertex. Thus in this case $|S|=2$.
   \\
 \textbf{Case 2:} $S_{\overline{G}}=\emptyset$. Let $u,v\in S_G$. If $uv\notin E(G)$, then $u,\overline{u},\overline{v},v$ is a monophonic path in $\GG$. Since $S_{\overline{G}}=\emptyset$, $S_G$ must be a clique in $G$. Hence $|S|\leq \omega(G)$.\\
\textbf{Case 3:} $S_{G}=\emptyset$. In this case,  replacing $G$ by $\overline{G}$ in Case 1, we get that $|S|\leq \omega(\overline{G})=\alpha(G)$. 
 \qed
 
 Next, we find the monophonic convexity number of disconnected graphs.
\begin{theorem}
\label{thm:convexity_disconnected}
Let $G$ be a disconnected graph with components $\{G_i\}_{i\in[r]}$, and $k$ be the order of a minimum component of $G$. Then
 $$\c_m (\GG)=
\begin{cases}
 2n(G)-1 \hspace{2cm}; $ if $k=1 \,,\\ \max\{\omega(G),\alpha(G)+2- (\min_{i\in [r]}\{\alpha(G_i)\})\hspace{0.3cm}; $ otherwise$\,. 
\end{cases}$$

\end{theorem}
\proof If $k=1$, then the single vertex $v$ in a minimum component of $G$ will be a simplicial vertex of $\GG$. Hence $V(\GG)\setminus \{v\}$ is a convex set of $\GG$ and so $\c_m (\GG)=2n(G)-1$. Now suppose that $k\geq 2$. Let $S$ be a monophonic convex set of $\GG$ with $|S|=\c_m (\GG)$. Fix $S_G=S\cap V(G)$ and $S_{\overline{G}}=S\cap V(\overline{G})$. Since $k\geq 2$, we have $|S|\geq 2$. Consider the following cases. \\
\textbf{Case 1:} $S_{\overline{G}}=\emptyset$. Let $u,v\in S_G$. If $uv\notin E(G)$, then $u,\overline{u},\overline{v},v$ is a monophonic path in $\GG$. This shows that $\overline{u},\overline{v}\in S$, which is impossible. Thus $S_G$ must be a clique in $G$. Hence $|S|\leq \omega(G)$. \\
\textbf{Case 2:} $S_{G}=\emptyset$. In this case, replacing $G$ by $\overline{G}$ in Case 1, we get $|S|\leq \alpha(G)$.\\
\textbf{Case 3:} $S_G \neq \emptyset$ and $S_{\overline{G}}  \neq \emptyset$. Let $u\in S_G$ and $\overline{v}\in S_{\overline{G}}$. We first prove the following  two claims. \\
\textbf{Claim 1:} 
If $u$ and $v$ are in the same component of $G$,  say in $G_1$, then \begin{equation}
\label{eqn1}
    \overline{S_G\cap G_1}=S_{\overline{G}}\cap \overline{G_1}.
\end{equation} Let $x\in S_G\cap G_1$.  Let $\overline{P}$ be a $\overline{v},\overline{x}$-monophonic path in $\overline{G}$. Then $\overline{P}$ together with the edge $\overline{x}x$ is a $x,\overline{v}$-monophonic path containing $\overline{x}$ in $\GG$. Hence $\overline{x}\in S$ and so $\overline{S_G\cap G_1}\subseteq S_{\overline{G}}\cap \overline{G_1}$. On the other hand, let $\overline{y}\in S_{\overline{G}}\cap \overline{G_1}$ and let $P$ be a $u,y$-monophonic path in $G_1$. Then $P$ together with the edge $y\overline{y}$ is a $u,\overline{y}$-monophonic path containing $y$ in $\GG$. Hence $y\in S_G$. This shows that $S_{\overline{G}}\cap \overline{G_1}\subseteq \overline{S_G\cap G_1}$.\\
\textbf{Claim 2:} For any component, $G_i$ of $G$, with $S_{G}\cap G_i\neq \emptyset$ and $S_{\overline{G}}\cap \overline{G_i}\neq \emptyset$,
\begin{equation}
    \label{eqn2}
    {\rm either\ } |S_G\cap G_i |= 1 {\rm \ or\ } |S_{\overline{G_i}}\cap \overline{G_i} |= 1. 
\end{equation}
Suppose that there exist distinct vertices $x_1, x_2$ in $S_G\cap G_i$ and distinct vertices $\overline{y_1}$ and $\overline{y_2}$ in $S_{\overline{G}}\cap \overline{G_i}$. Without loss of generality, we may assume that $i=1$; and also we may assume that $x_1\neq y_1$. If $x_1$ and $y_1$ are adjacent vertices in $G$, then $V(\overline{G})\setminus V(\overline{G_1})\subseteq S_{\overline{G}}$. Since $\overline{x_1},\overline{y_1}\in S_{\overline{G}}$ by Claim 1 and so $V(\overline{G})\setminus V(\overline{G_1})\subseteq J_{\overline{G}}[\overline{x_1},\overline{y_1}]$. Since the order of any component is greater than 2, we get $V(\overline{G})\subseteq S$. Now recall that for any pair of adjacent vertices $z$ and $w$ in $V(G_i)$, with $i\in[r]$, $\overline{z},z,w,\overline{w}$ is an $m$-path. Since for any $z\in V(G_i)$, $\deg_{G}(z)\geq 1$, for any $i\in[r]$, $V(G_i)\subseteq S$ as $V(\overline{G_i})\subseteq S$. Hence $V(\GG)\subseteq S$. This is not possible, as $S$ is a proper subset of $\GG$. Hence $S_G\cap G_1$ must be independent. But then $S$ cannot be convex, since $G_1$ is connected. Hence the claim follows.
  
  In the following, we prove that $|S_G|=1$. Assume the contrary, that there exist two distinct vertices $u_1,u_2\in S_G$. Suppose first that $u_1$ and $u_2$ are in the same component of $G$, say in $G_1$. Then by Claim 1 and Claim 2, we have that $S_{\overline{G}}\cap V(\overline{G_1})=\emptyset$. Choose $\overline{x}\in S_{\overline{G}}$. Then $\overline{x}\notin \overline{G_1}$ and $u_1,\overline{u_1},\overline{x}$ is a monophonic path in $\GG$. Thus this case is not possible as $S_{\overline{G}}\cap \overline{G_1}= \emptyset$. Thus $u_1$ and $u_2$ must be from different components, say $u_1\in G_1$ and $u_2\in G_2$. Let $x$ be any vertex in $G_2$ distinct from $u_2$ and let $P$ be a monophonic path connecting $x$ and $u_2$. Then the path $u_1,\overline{u_1},\overline{x},x$ together with the path $P$ is a monophonic path connecting $u_1$ and $u_2$ containing $x$ and $\overline{x}$. Hence $x,\overline{x}\in S$. Also by Claim 1, $\overline{u_2}\in S$. This contradicts Claim 2. Thus $|S_G|= 1$.\\
  \textbf{Claim 3:} $S_{\overline{G}}$ is a clique.  If possible suppose that $\overline{x},\overline{y} \in S_{\overline{G}}$, for some $i\in [r]$ with  $x$ and $y$ adjacent in $G$. Then the path $\overline{x},x,y,\overline{y}$ is a monophonic path in $\GG$. Thus $x,y\in S_G$. This is impossible by Claims 1 and Claim 2. Thus $S_{\overline{G}}$ is a clique.

 Let $ S_G=\{u\}$ and without loss of generality we may assume that $u\in V(G_1)$. Also, let $l=\min_{j\in[r]}\{\alpha(G_j)\}$. Then by Claim 1 and Claim 2, $|S_{\overline{G}}\cap \overline{G_i}|= 1$. Thus $|S|\leq \alpha(G)+2-(\alpha(G_i))$ . This shows that $|S|\leq \alpha(G)+2-l$.

 Finally, we show that $\GG$ contains a $m$-convex set of order $\alpha(G)+2- l$. 
For that, let $I$ be an independent set in $G$ with $|I|=\alpha(G)$. Choose a vertex $u$ in $I$ such that  $u\in V(G_i)$ and $\alpha(G_i)=l$. Consider the set $H=\{u,\overline{u}\} \cup C$, where $C=\overline{I} \setminus V(\overline{G_i}) $. Then $|H|=\alpha(G)+2-l$. We claim that $H$ is a $m$-convex set of $\GG$. We remark that $\overline{I}$ is $m$-convex since it is a clique in $\GG$. Now let $\overline{x}$ be a vertex in $\overline{I}$ which is distinct from $\overline{u}$. Then $x$ and $u$ are in different components of $G$ and $u,\overline{u},\overline{x}$ is the unique $m$-path connecting $u$ and $\overline{x}$. This shows that the set $H$ must be $m$-convex in $\GG$. 
This in turn implies that $\c_m(\GG)=\max\{\omega(G), \alpha(G)+2-l\}$.
\qed

\section{Monophonic number}
\label{monophonic}

In this section, we are concerned with the monophonic number of complementary prisms. As done in Section~\ref{convexity}, the results are separated according to the connectedness of $G$. We begin when $G$ is a disconnected graph.

\begin{theorem}
\label{thm:monophonic_disconnected}
Let $G$ be a disconnected graph with $r \geq 2$ connected components. Then $\m(\GG)= r$.
\end{theorem}

\begin{proof}
Let $G_1, G_2, \dots, G_r$ be the components of $G$. 
Consider $S = \{v_1, v_2, \dots, v_r\}$ such that $v_i$ is any vertex of $V(G_i)$, for every $i \in [r]$. We show that $S$ is a monophonic set. 
Since $G_i$ is connected, there is a monophonic path $P: v_i, \dots, u$ in $G_i$, for every $u \in V(G_i)$. ($P$ can be the shortest path between $v_i$ and $u$, for instance). Then, $P$ together with the edges  $u\bar{u}$, $\bar{u}\bar{v}_j$, $\bar{v}_jv_j$ is a   $v_i,v_j$-monophonic path containing $\{u, \overline{u}\}$, for every $u \in V(G_i)$, for every $1 \leq i < j \leq r$. Then, $J_{\GG}[S] = V(\GG)$ and $\m(\GG) \leq r$. 

On the other hand, to show that $r$ is the minimum order of a monophonic set of $\GG$, we prove Claim~1. 

\smallskip
\noindent\textbf{Claim 1:} \emph{If $S$ is a monophonic set of $\GG$, then $V(G_i\overline{G}_i) \cap S \neq \emptyset$, for every $i \in [r]$.}

\smallskip

Let $t \in \{0, 1, \dots, r\}$ be the number of trivial components in $\GG$. If $t \geq 1$ and $1 \leq i \leq t$, the conclusion that $V(G_i\overline{G}_i) \cap S \neq \emptyset$ is direct, since $v_i$ is simplicial. Then, consider that $t + 1 \leq i \leq r$, for $t \in \{ 0, 1, \dots, r-1 \}$.

For a contradiction, suppose that $V(G_i\overline{G}_i) \cap S = \emptyset$, for some $t + 1 \leq i \leq r$. Let $v \in V(G_i)$. Since $S$ is a monophonic set, $v$ lies in a $x, y$-monophonic path $P$ in $\GG$, for $x,y \in V(\GG) \setminus V(G_i\overline{G}_i)$. Since $G_i$ is a nontrivial component and $P$ is an $m$-path, we have that $|V(P)| \geq 6$ as well as $|V(P) \cap V(\overline{G}_i)| \geq 2$. Then, let $P: x, \dots, x', x'', \dots, y'',y', \dots y$ with $x', y' \in V(\overline{G}) \setminus V(\overline{G}_i)$, $x'', y'' \in V(\overline{G}_i)$, and $x'x'', y'y'' \in E(\overline{G})$.
Notice that $x'y''$ is an edge in $\overline{G}$. Since $x'$ and $y''$ are not consecutive vertices in $P$, then $x'y''$ is a chord, a contradiction.

\medskip

By Claim~1, we conclude that $\m(\GG) \geq r$ and the proof is done.
\qed \end{proof}

In the following, we consider both $G$ and $\overline{G}$ be connected graphs. Recall that $n(G) \geq 4$.

\begin{lemma}\label{lemma:monophonic_diam3}
Let $G$ be a connected graph such that $\overline{G}$ is connected. If $\diam(G) \geq 3$, then $\m(\GG)= 2$.
\end{lemma}

\begin{proof}
Let $u,v \in V(G)$ such that $d_G(u,v) = 3$ and $S = \{u,v\}$. Since $N_G(u) \cap N_G(v) = \emptyset$, we have that, for every $x \in N_G(u)$, $\overline{x} \in N_{\overline{G}}(\overline{v})$. Then, $u,x,\overline{x},\overline{v},v$ is a monophonic path and $x,\overline{x} \in J_{\GG}[u,v]$. Similarly, for every $y \in N_G(v)$, it holds that $v,y,\overline{y},\overline{u},u$ is a monophonic path and $y,\overline{y} \in J_{\GG}[u,v]$.

Next, we show that for every $z \in V(G) \setminus (N_G(u) \cup N_G(v))$, the vertices $z, \overline{z}$ lie in an $u,v$-monophonic path. For that, we choose $w \in \{u,v\}$ such that $d_G(z,w)$ is minimum and let $P$ be a shortest $z,w$-path in $G$. Let $w' \in \{u,v\} \setminus \{w\}$. We show that $P$ together with the edges $\{z\overline{z},\bar{z}\bar{w}',\overline{w}'w'\}$ is an $u,v$-monophonic path in $\GG$. 

Given that $P$ is a shortest path in $G$, as well as in $\GG$, it holds that $P$ is a monophonic path in $\GG$. Notice that $P$ with $\{z\overline{z},\bar{z}\bar{w}'\}$ is also a monophonic path in $\GG$, since $\bar{z} \notin N_{\GG}(P \setminus \{z\})$ and $\bar{w}' \notin N_{\GG}(P)$. Recall that $w' \notin N_{\GG}(\overline{z})$, then, remains to show that $w' \notin  N_{\GG}(P)$. That conclusion holds since $N_G(u) \cap N_G(v) = \emptyset$ and $P$ is a shortest $z,w$-path in $G$ implies that $V(P) \cap N_G(w') = \emptyset$. 

Hence, $z, \overline{z} \in J_{\GG}[u,v]$ and $S$ is a monophonic set of $\GG$.
\qed
\end{proof}

Recall that Lemma~\ref{lemma:monophonic_diam3} treats the case  $\diam(G) \geq 3$ or $\diam(\overline{G}) \geq 3$. Then, from now on, we consider $\diam(G) = \diam(\overline{G}) = 2$. With such a restriction, the following remark will be useful.

\begin{remark}[\cite{chellaram-2019}]\label{rmk:diam2}
Let $G$ be a graph. If $\diam(G) = \diam(\overline{G}) = 2$, then for every edge $uv \in E(G)$, $ecc_G(u) = ecc_G(v) = 2$ and $V(G) \setminus (N_G(u) \cup N_G(v) ) \neq \emptyset$.
\end{remark}

\begin{lemma}\label{lemma:monophonic_c5}
Let $G$ be a graph. If $G = C_5$, then  $\m(\GG)= 3$.
\end{lemma}

\begin{proof}
Let $G$ be a cycle graph on $5$ vertices, with $V(G) = \{u_1, u_2, \dots, u_5\}$ and $E(G) = \{u_iu_{i+1} : 1 \leq i \leq 4\} \cup \{ u_5u_1 \}$. Consider $S = \{u_1,u_4, \overline{u}_5 \}$. 
Since $u_1, u_5, u_4$ is an $m$-path in $G$, $\overline{u}_5 \in S$, and by Proposition~\ref{prop:monophonic_p4}, for every $i \in [4]$, $u_i, \overline{u}_i \in J_{\GG}[u_1,u_4]$, we conclude that $S$ is a monophonic set of $\GG$.

Next, we show that $\m(\GG) \geq 3$. Suppose, by contradiction, that $\m(\GG) = 2$. Let $S = \{ x,y \}$ be a monophonic set of $\GG$. Since $|V(G)| > 1$, we have that $xy \notin E(\GG)$. We consider two cases: \emph{(i)} $x,y \in V(G)$ and \emph{(ii)} $x \in V(G)$ and $y \in V(\overline{G})$.

\emph{(i)} Since for every pair of nonadjacent vertices $x,y \in V(G)$, $d_G(x,y) = 2$, we may assume w.l.o.g. that $x = u_1$ and $y = u_4$. Since $S$ is an $m$-set of $\GG$, $\overline{u}_5$ lies in a monophonic path $P: u_1, \dots, \overline{u}_5, \dots, u_4$. It is clear that $u_5 \notin V(P)$, otherwise an edge from $\{u_1,u_5, u_4u_5, u_5\overline{u}_5\}$ would be a chord. This implies that $\overline{u}_2\overline{u}_5, \overline{u}_3\overline{u}_5 \in E(P)$. If $u_2 \in V(P)$, then $u_1u_2, u_2\overline{u}_2 \in E(P)$. Then, $u_3 \notin V(P)$, consequently $u_4\overline{u}_4 \in E(P)$. But $\overline{u}_2\overline{u}_4$ is a chord in $P$, a contradiction. So, consider that $u_2 \notin V(P)$. Then, $u_1\overline{u}_1, \overline{u}_1\overline{u}_3 \in E(P)$. This implies that $\overline{u}_4 \notin V(P)$, consequently $u_3 \in V(P)$ as $u_4u_3, u_3\overline{u}_3 \in E(P)$. But $\overline{u}_1\overline{u}_3$ is a chord in $P$, contradiction. Hence, $\overline{u}_5 \notin J_{\GG}[u_1,u_4]$, a contradiction.

\emph{(ii)} Let $x \in V(G)$ and $y \in V(\overline{G})$ with $xy \notin E(G)$. Since $d_{\GG}(x,y) = 2$, we may assume w.l.o.g. that $x = u_1$ and $y = \overline{u}_2$. Since $S$ is an $m$-set of $\GG$, $u_3$ lies in a monophonic path $P: u_1, \dots, u_3, \dots, u_4$. Notice that $u_2 \notin V(P)$, otherwise an edge from $\{u_1,u_2, u_2u_3, u_2\overline{u}_2\}$ would be a chord. Then $u_3u_4, u_3\overline{u}_3 \in E(P)$. If $u_5 \in V(P)$, then $u_1u_5, u_4u_5 \in E(P)$. Consequently $\overline{u}_5 \notin V(P)$ and $\overline{u}_2\overline{u}_4 \in E(P)$. But in this case, $u_4\overline{u}_4$ is a chord of $P$, a contradiction. Then, consider $u_5 \notin V(P)$. This implies that $u_1\overline{u}_1 \in E(P)$, but $\overline{u}_1\overline{u}_3$ is a chord of $P$, also a contradiction. Therefore, $\m(\GG) \geq 3$. \qed
\end{proof}

\begin{lemma}
\label{lemma:monophonic_diam2}
Let $G$ be a graph such that $\diam(G) = \diam(\overline{G}) = 2$. If $G \neq C_5$, then $\m(\GG)= 2$.
\end{lemma}

\begin{proof}
Let $G \neq C_5$ be graph with $\diam(G) = \diam(\overline{G}) = 2$. We present two claims for the proof.

\smallskip
\noindent\textbf{Claim 1:} \emph{If there exists $uv \in E(G)$ such that $|N_G(u) \setminus N_G(v)|\geq 2$, then $\m(\GG)= 2$.}

\smallskip
Let $uv \in E(G)$ such that $N_G(u) \setminus N_G(v) = \{u_1,u_2\}$ and $S = \{u,\overline{v}\}$.
Given that $u,v,\overline{v}$ is an $m$-path, the conclusion that $v \in J_{\GG}[S]$ is immediate. By Remark~\ref{rmk:diam2}, we know that there exists $x \in V(G) \setminus (N_G(u) \cup N_G(v))$. Then, $\overline{u},\overline{v} \in N_{\overline{G}}(\overline{x})$. This implies that $u,\overline{u},\overline{x},\overline{v}$ is an $m$-path and $\overline{u} \in J_{\GG}[S]$. 

Recall that $u_1,u_2 \notin N_G(v)$. Then $\overline{u}_1, \overline{u}_2 \in N_{\overline{G}}(\overline{v})$. Consequently, $u, u_i, \overline{u}_i, \overline{v}$, for every $i \in \{1,2\}$, is a monophonic path in $\GG$ and $u_i,\overline{u}_i \in J_{\GG}[S]$. This in turn implies that $(N_G(u) \setminus N_G(v)) \cup (\overline{N_G(u) \setminus N_G(v)}) \subseteq J_{\GG}[S]$.

Now, let $v_1 \in N_G(v) \setminus N_G(u)$. Notice that 
there is no monophonic $u,v_1$-path containing both $u_1, u_2$, otherwise either $uu_1$ or $uu_2$ would be a chord. In addition, since $d_G(v_1,u_i) \leq 2$, for every $i \in \{1,2\}$, a shortest $v_1, u_i$-path  in $G$ does not contain $v$. Then, we select $i \in \{1,2\}$ such that $v_1, \dots, u_i$ is a shortest path in $G$ of minimum length.
Thus, $u, u_i, \dots, v_1, \overline{v}_1, \overline{u}_{3-i}, \overline{v}$ is a monophonic $u, \overline{v}$-path in $\GG$. Consequently $v_1, \overline{v}_1 \in J_{\GG}[S]$. Furthermore, this implies that $(N_G(v) \setminus N_G(u)) \cup (\overline{N_G(v) \setminus N_G(u)}) \subseteq J_{\GG}[S]$.
 
Next, let $y \in N_G(u) \cap N_G(v)$. By Remark~\ref{rmk:diam2}, there exists $w \in V(G) \setminus (N_G(u) \cup N_G(y))$. If $w \notin N_G(v)$, then $\overline{w} \in N_{\overline{G}}(\overline{v})$, consequently $u, y, \overline{y}, \overline{w}, \overline{v}$ is an $m$-path in $\GG$. Otherwise, $vw$ being an edge in $G$, Remark~\ref{rmk:diam2} implies that there exists $w' \in V(G) \setminus (N_G(v) \cup N_G(w))$. Thus, $\overline{v},\overline{w} \in N_{\overline{G}}(\overline{w}')$. Hence, $u, y, \overline{y}, \overline{w},\overline{w}', \overline{v}$ is an $m$-path in $\GG$.
In both cases, we conclude that $y,\overline{y} \in J_{\GG}[S]$, which implies that $(N_G(u) \cap N_G(v)) \cup (\overline{N_G(u) \cap N_G(v)}) \subseteq J_{\GG}[S]$.

Finally, let $x \in V(G) \setminus (N_G(u) \cup N_G(v))$. Since $d_G(u,x) = 2$, there exists $u' \in N_G(x)$ such that $u'x \in E(G)$. Then, $u,u', x, \overline{x}, \overline{v}$ is an $m$-path in $\GG$. Thus $x, \overline{x} \in J_{\GG}[S]$ and $(V(G) \setminus (N_G(u) \cup N_G(v))) \cup (\overline{V(G) \setminus (N_G(u) \cup N_G(v))}) \subseteq J_{\GG}[S]$.

By the above, $S = \{u,\overline{v}\}$ is an $m$-set of $\GG$, therefore $\m(\GG) = 2$.

\medskip
Recall that Claim~1 considers the existence of $uv \in E(G)$ such that $|N_G(u) \setminus N_G(v)|\geq 2$. Then, for the rest of the proof, we consider that there is no such an edge, i.e., for every $uv \in E(G)$, $|N_G(u) \setminus N_G(v)| \leq 1$. We proceed with Claim~2.

\smallskip
\noindent\textbf{Claim 2:} \emph{If there exists a clique of order $3$ in $G$, then $\m(\GG)= 2$.}

\smallskip

Let $C = \{a,b,c\}$ be a clique in $G$. Since $ab \in E(G)$, Remark~\ref{rmk:diam2}, implies that there exists $x \in V(G) \setminus (N_G(a) \cup N_G(b))$. Since $\diam(G) = 2$, there exists $y \in N_G(a) \cup N_G(b)$ such that $xy \in E(G)$. 

Recall that $ax,bx \notin E(G)$. If $y = c$, then $cx$ is an edge with $|N_G(c) \setminus N_G(x)|\geq 2$ and the conclusion follows by Claim~1.
Otherwise, w.l.o.g., suppose that $y \in N_G(a) \setminus N_G(b)$. If $cy \notin E(G)$, given that $by \notin E(G)$, we have that $ay$ is an edge with $|N_G(a) \setminus N_G(y)|\geq 2$ and, again, the conclusion follows by Claim~1. So, suppose that $cy \in E(G)$. Since $ax,cx \notin E(G)$, $yx$ is an edge with $|N_G(y) \setminus N_G(x)|\geq 2$. Once applying Claim~1, the conclusion holds.

\medskip

By Claim~2, we may assume that $G$ contains no $K_3$ as an induced subgraph, that is, $G$ is \emph{triangle-free}. Recall that for every $uv \in E(G)$, $|N_G(u) \setminus N_G(v)| \leq 1$. If $G$ is triangle-free and there exists $uv \in E(G)$, with $|N_G(u) \setminus N_G(v)| = 0$, then $deg_G(u) = 1$. Since $\diam(G) = 2$, for every $x \in V(G) \setminus \{u,v\}$, $d_G(u,x) = 2$. But this implies that $N_G(u) \cup N_G(v) = V(G)$, which contradicts Remark~\ref{rmk:diam2}. Then $G$ has no pendant vertices. Finally, since for every $uv \in E(G)$, $|N_G(u) \setminus N_G(v)| = 1$, $\diam(G) = \diam(\overline{G)} = 2$ implies that $G = C_5$, which was proved in Lemma~\ref{lemma:monophonic_c5}.
\qed \end{proof}

Putting everything together, we close this section with Theorem~\ref{thm:monophonic_connected}.

\begin{theorem}
\label{thm:monophonic_connected}
Let $G$ be a connected graph such that $\overline{G}$ is connected, then  
$$\m(\GG)=
\begin{cases}
 3 \hspace{0.4cm}; $ if $G = C_5 \,, \\ 
 2\hspace{0.4cm}; $ otherwise$\,. 
\end{cases}$$
\end{theorem}

\begin{proof}
The conclusion follows by Lemmas~\ref{lemma:monophonic_diam3},~\ref{lemma:monophonic_c5}, and~\ref{lemma:monophonic_diam2}.
\qed \end{proof}

\section{Monophonic hull number}
\label{hull}
In this section, we determine the monophonic hull number of complementary prisms. As mentioned in Introduction, in~\cite{dourado-2010} the authors prove that the monophonic hull number of an arbitrary graph can be computed in polynomial time. Although they presented an $O(n^3m)$-time algorithm for computing such a parameter on a graph $G$ with $n$ vertices and $m$ edges, we present here closed formulas for $\h_m(\GG)$ that depend solely on the number of trivial components of $G$. This leads to a linear time algorithm for computing the hull number of complementary prisms.

\begin{theorem}
\label{thm:hull-disconnected}
Let $G$ be a disconnected graph with $t$ trivial components. Then
 $$\h_m (\GG)=
\begin{cases}
 2 \hspace{1cm}; $ if $t=0 \,, \\ 
 t+1\hspace{0.4cm}; $ otherwise$\,. 
\end{cases}$$
\end{theorem}

\begin{proof}
Suppose that $G$ has no trivial components. Denote by $G_1, G_2, \dots, G_r$ the components of $G$. Let $S = \{\overline{u},\overline{v}\}$ such that $uv \in E(G_1)$. We show that $S$ is an $m$-hull set of $\GG$. Since $\bar{u} \bar{v} \notin E(\overline{G})$ and, for every $w \notin V(G_1)$, $\bar{u}\bar{w}, \bar{v}\bar{w} \in E(\overline{G})$, we have that $\bar{u}, \bar{w}, \bar{v}$ is a monophonic path in $\overline{G}$. Then $\bigcup_{i = 2}^{r} V(\overline{G}_i)  \subseteq J_{\GG}[\overline{u},\overline{v}]$. Since $G_2$ is nontrivial, there exists $x,y \in V(G_2)$ such that $xy \in E(G)$. By the same argument, we get $V(\overline{G}_1)  \subseteq J_{\GG}[\overline{x},\overline{y}]$. Finally, given that $V(\overline{G}) \subseteq [S]_{\GG}$ and, for every $i \in [r]$, there exist $x,y \in V(G_i)$ such that $xy \in E(G)$, $x,y \in J_{\GG}[\overline{x},\overline{y}]$. By applying Claim~1 of proof of Theorem~\ref{thm:convexity_disconnected} we get that $V(G) \subseteq [S]_{\GG}$.

Now, suppose that $G$ has $t \geq 1$ trivial components. Let $G_1, G_2, \dots, G_r$ denote the components of $G$, where $V(G_i) = \{v_i\}$, for $1 \leq i \leq t$, and $n(G_i) \geq 2$, for $t+1 \leq i \leq r$. Notice that $v_i$ is a simplicial vertex, for every $1 \leq i \leq t$. Then, Remark~\ref{rmk:simplicial} implies that $v_1, v_2, \dots, v_t$ belongs to any $m$-hull set of $\GG$ and $\h_m (\GG) \geq t$. For every $i,j \in [t]$, we know that $\overline{v}_i\overline{v}_j$ is a chord for every $v_i,v_j$-path in $\GG$. Then, $[\{v_1, v_2, \dots, v_t\}]_{\GG} = \{v_i, \overline{v}_i : 1 \leq i \leq t\}$, which implies that $\h_m (\GG) \geq t+1$.

Let $u$ be any vertex in a nontrivial component of $G$, say $G_{t+1}$. We show that $S = \{u, v_1, v_2, \dots, v_t\}$ is an $m$-hull set of $\GG$. We know that $u, \overline{u}, \overline{v}_i, v_i$ (resp. $u', \overline{u}', \overline{v}_i, v_i$) is a monophonic path between $u$ (resp. $u'$, for $u'\in N_G(u)$) and $v_i$, for every $i \in [t]$. Then $(\{v_i, \overline{v}_i : 1 \leq i \leq t\} \cup N_G[u] \cup \overline{N_G[u]}) \subseteq [S]_{\GG}$. Let $u'\in N_G(u)$. Since $\overline{u}\overline{u'} \notin E(G)$, we have that $\bigcup_{i = t+2}^{r} V(\overline{G}_i)  \subseteq J_{\GG}[\overline{u},\overline{u'}]$. Given that $G_{t+2}$ is nontrivial, $V(\overline{G}_{t+1}) \subseteq J_{\GG}[V(\overline{G}_{t+2})]$, which implies that $V(\overline{G}) \subseteq [S]_{\GG}$. Finally, using a similar argument as in the case of $t = 0$, we conclude that $V(G) \subseteq [S]_{\GG}$.
\qed \end{proof}

\begin{theorem}
\label{thm:hull-connected}
Let $G$ be a connected graph such that $\overline{G}$ is connected, then  $\h_m (\GG)= 2$.
\end{theorem}

\begin{proof}
Since $\h_m(\GG)$ is upper bounded by $\m(\GG)$ (cf. Remark~\ref{rmk:hm_bound}), when $G \neq C_5$,  Theorem~\ref{thm:monophonic_connected} imples that $\h_m (\GG)= 2$. Then, it is enough to prove that $\h_m (C_5\overline{C}_5)= 2$. 

Let $G$ be a cycle graph on $5$ vertices, with $V(G) = \{u_1, u_2, \dots, u_5\}$ and $E(G) = \{u_iu_{i+1} : 1 \leq i \leq 4\} \cup \{ u_5u_1 \}$. Consider $S = \{u_1,u_4\}$. 
Proposition~\ref{prop:monophonic_p4} implies that, for every $i \in [4]$, $u_i, \overline{u}_i  \in J_{\GG}[S] \subseteq [S]_{\GG}$. Since $u_1, u_5, u_4$ is an $m$-path in $G$, $u_5 \in J_{\GG}[S] \subseteq [S]_{\GG}$. Finally, $\overline{u}_2, \overline{u}_5, \overline{u}_3$ is an $m$-path in $\GG$, then $\overline{u}_5 \in J_{\GG}[\overline{u}_2, \overline{u}_3] \subseteq J_{\GG}^2[S] \subseteq [S]_{\GG}$. Therefore $S$ is an $m$-hull set of $\GG$ and $\h_m(\GG)= 2$.
\qed
\end{proof}



\end{document}